\documentclass[a4paper,11pt]{article}
\usepackage[mathscr]{eucal}
\usepackage{amssymb}
\usepackage{latexsym}
\usepackage{amsthm}
\usepackage{amsmath}
\usepackage[dvips]{graphicx}
\usepackage{psfrag}
\usepackage{a4wide}

\DeclareMathOperator{\Rep}{{\rm Rep}}
\DeclareMathOperator{\R}{\mathbb{R}}

\newtheorem{theorem}{Theorem}[section]

\newtheorem{lemma}[theorem]{Lemma}

\theoremstyle{definition}

\newtheorem{remark}[theorem]{Remark}

\makeatletter
\@addtoreset{equation}{section}

\makeatother

\title{Complex spherical codes with two inner products
}
\author{Hiroshi Nozaki\thanks{Department of Mathematics Education, Aichi University of Education, Kariya, 448-8542, Japan. \emph{E-mail address}: \texttt{hnozaki@auecc.aichi-edu.ac.jp}}, Sho Suda\thanks{Department of Mathematics Education, Aichi University of Education, Kariya, 448-8542, Japan. \emph{E-mail address}: \texttt{suda@auecc.aichi-edu.ac.jp}}}
\begin{document}
\maketitle

\renewcommand{\thefootnote}{\fnsymbol{footnote}}

\begin{abstract}
A finite set $X$ in a complex sphere is called a complex spherical $2$-code if
the number of inner products between two distinct vectors in $X$ is equal to $2$. 
In this paper, we characterize the tight complex spherical $2$-codes 
by doubly regular tournaments or skew Hadamard matrices. 
We also give certain maximal 2-codes relating to skew-symmetric $D$-optimal designs. 
To prove them, we show the smallest embedding dimension of a tournament into a complex sphere by the multiplicity of the smallest or second-smallest eigenvalue of the Seidel matrix. 
\end{abstract}
\section{Introduction}
Let $X$ be a finite set of points on the  complex unit
sphere $\Omega(d)$ in $\mathbb{C}^d$. The {\it angle set} $A(X)$
is defined to be 
\[
A(X)=\{x^*y \mid  x,y \in X, x \ne y \}, 
\]
where $x^*$ is the transpose conjugate of a column vector $x$. 
A finite set $X$ is called a {\it complex spherical $s$-code} 
if $|A(X)|=s$ and $A(X)$ contains an imaginary number. The value $s$ is called the {\it degree} of $X$.  
For $X, X' \subset \Omega(d)$, we say that $X$ is  
{\it isomorphic} to $X'$ if there exists a unitary transformation
from $X$ to $X'$. 
An $s$-code $X\subset \Omega(d)$ is said to be {\it largest} if $X$ has 
the largest possible cardinality in all $s$-codes in $\Omega(d)$.  
One of major problems on $s$-codes 
is to classify largest $s$-codes for given $s$ and $d$. 

We will survey Euclidean finite sets with only $s$ distances. 
For $X \subset \mathbb{R}^d$, we define
\[
D(X)=\{d(x,y) \mid x,y \in X, x\ne y\},
\] 
where $d(x,y)$ is the Euclidean distance of $x$ and $y$. A finite set $X$ is called an {\it $s$-distance set} if $|D(X)|=s$ holds. 
We have an upper bound for the size of an $s$-distance set in $\R^d$, namely $|X| \leq \binom{d+s}{s}$ \cite{BBS}. Clearly
the largest $1$-distance set in $\R^d$ is the regular simplex for any $d$. 
Largest $2$-distance sets in $\R^d$ are classified for $d\leq 7$ \cite{ES66,L97}. 
Largest $s$-distance sets in $\R^2$ are classified for $s\leq 5$ \cite{EF96,S04,S08}.  
The largest $3$-distance set in $\R^3$ is the vertex set of the icosahedron \cite{Spre}. 
The classification of largest $s$-distance sets is still open for others $(s,d)$. 
A largest $2$-distance set in $\R^8$ is given in \cite{L97}, and it attains the upper bound. 

A spherical $s$-distance set particularly deserves attention because of the connection to association schemes or  
spherical $t$-designs (see \cite{DGS77, BB} for details). 
A subset $X$ of $S^{d-1}$ is called 
a {\it spherical $t$-design} if for any polynomial $f$ in $d$ variables of degree at most $t$, the following equality 
holds: 
\[
\frac{1}{|S^{d-1}|}\int_{S^{d-1}} f(x) dx= \frac{1}{|X|} \sum_{x \in X} f(x),
\]
where $|S^{d-1}|$ is the volume of $S^{d-1}$.   
 If a spherical $t$-design $X$ of degree $s$ satisfies  
$t\geq 2s-2$, then $X$ has the structure of a $Q$-polynomial association scheme \cite{DGS77}. The size of an $s$-distance set in
$S^{d-1}$ is smaller than or equal to $\binom{d+s-1}{s}+\binom{d+s-2}{s-1}$ \cite{DGS77}. An $s$-distance set $X$ is said to be {\it tight} if $X$ attains this bound. A tight $s$-distance set becomes a minimal spherical $t$-design and satisfies $t=2s$ \cite{DGS77}. The classification of tight $s$-distance sets is one of the most interesting problems, and this has been solved except for $s=2$ \cite{BD79}. 
A largest $2$-distance set on $S^{d-1}$ is determined for 
$d\leq 93$ ($d\ne 46,78$) \cite{M09,AW13}. 
A largest $3$-distance set on $S^{d-1}$ is determined for $d=2,3,8,22$ \cite{Spre,MN11}.

A simple graph $G=(V,E)$ is {\it representable} in $\R^d$ if  
there is an embedding $\sigma:V\rightarrow \R^d$   such that 
\[
d(\sigma(a),\sigma(b))=\begin{cases}
\alpha \text{ if $(a,b) \in E$}, \\
\beta \text{ otherwise},
\end{cases}
\]
for some $\alpha,\beta\in\R$.
For a simple graph $G$, Roy \cite{R10} gave an explicit expression of the minimal dimension $d$ such that 
$G$ is representable in $\R^d$ in terms of the multiplicity of the smallest or second-smallest eigenvalue of $A$.    This embedding of a graph is useful for the classification of $2$-distance sets \cite{ES66,L97}. 
 
Roy and Suda \cite{RSX} gave the complex analogue of the spherical $s$-distance set theory. 
Complex spherical $s$-codes are closely related to complex spherical designs or non-symmetric association schemes. In this paper, we consider a complex spherical $2$-code $X\subset \Omega(d)$. If $X$ satisfies $A(X) \subset \R$, then the Gram matrix of $X$ is real, and $X$ can be embedded into $\R^d$. 
We may assume $A(X)$ contains an imaginary number $\alpha$, and $A(X)=\{\alpha, \overline{\alpha}\}$, where $\overline{\alpha}$ is the conjugate of $\alpha$. We have a natural upper bound \cite{RSX}: 
\begin{equation} \label{eq:2_bound}
|X| \leq \begin{cases}
2d+1 & \text{if $d$ is odd},\\
2d & \text{if $d$ is even}. 
\end{cases}
\end{equation}
A $2$-code $X$ is said to be {\it tight} if $X$ attains the bound \eqref{eq:2_bound}. 
This is known as the {\em absolute bound}.

A {\it tournament} is a directed graph obtained by assigning a direction for each edge in an undirected complete graph.
Formally, a tournament is a pair $(V,E)$ such that  the vertex set $V$ is a finite set and the edge set $E\subset V\times V$ satisfies $E\cap E^T=\emptyset$ and $E\cup E^T\cup \{(x,x)\mid x\in V\}=V\times V$, where $E^T:=\{(x,y)\mid (y,x)\in E\}$. 
A complex spherical $2$-code $X$ has the structure of a tournament $(X,E)$, where $E=\{(x,y)\in X\times X \mid x^*y=\alpha \}$. A tournament $(V,E)$ is {\it representable in $\Omega(d)$} 
if there exists 
a mapping $\varphi$ from $V$ to $\Omega(d)$ 
such that for all distinct $x,y \in V$, 
\[
\varphi(x)^* \varphi(y)=\begin{cases}
\alpha \text{ if $(x,y) \in E$}, \\
\overline{\alpha} \text{ if $(y,x) \in E$},
\end{cases}
\]
where
$\alpha$ is an imaginary number with ${\rm Im}(\alpha)>0$.  
Such a mapping $\varphi$ is said to be a {\it representation} of a tournament. 
We identify a representation with the image of the representation. 
Two tournaments $G=(V,E),G'=(V',E')$ are {\it isomorphic} if there is a bijection from $V$ to $V'$ such that $(x,y)\in E$ if and only if $(f(x),f(y))\in E'$.
For two tournaments $G$ and $G'$, if $G$ is not isomorphic to $G'$, then a representation of $G$ is not isomorphic to that of $G'$. 
Let $\Rep(G)$ denote the smallest $d$ such that $G$ is representable in $\Omega(d)$.
The {\it Seidel matrix} of $G$ is defined to be $
\sqrt{-1}(A-A^T)$, where $A$ is the adjacency matrix of $G$. 
In Section~\ref{sec:Crep}, we determine $\Rep(G)$ by the multiplicity of the smallest or second-smallest eigenvalue of 
the Seidel matrix of $G$. 

A tournament $G$ is said to be {\it doubly regular} if  the number of the neighbors of a vertex does not depend on the choice of the vertex and the number of the common neighbors of a pair of distinct vertices does not depend on the choice of the pair.
An $n\times n$ $(\pm1)$-matrix of $H$ is called a {\it skew Hadamard matrix} if $H+H^T=2I$ and $H H^T=nI$, where $I$ is the identity matrix.   
Let $X \subset \Omega(d)$ be a $2$-code, and $A$  the adjacency matrix of the tournament obtained from $X$. 
It is known that the existence of a doubly regular tournament of $4d+3$ vertices is equivalent to that of a skew Hadamard matrix of order $4d+4$ \cite{RB72}. 
In Section~\ref{sec:tight}, 
we give the following characterizations of tight $2$-codes and $2$-codes with $n=2d$ where $d$ is odd.   
\begin{enumerate}
\item For odd $d$, $X$ is a tight complex $2$-code if and only if $A$ is the adjacency matrix of a doubly regular tournament.
\item For even $d$, $X$ is a tight complex $2$-code if and only if $I+A-A^T$ is a skew Hadamard matrix. 
\item For odd $d$, $X$ is a complex $2$-code with $n=2d$ if and only if either $A$ is the adjacency matrix of an induced subgraph of a  doubly regular tournament by deleting a vertex, or its Seidel matrix $S$ satisfies that $S^2$ is permutationally similar to 
\begin{align*}
\begin{pmatrix}
k I+l J&0\\0&k I+l J
\end{pmatrix}, 
\end{align*}
for some positive integers $k,l$. \label{en:3}
\end{enumerate} 
We note that the last case in \eqref{en:3} includes skew-symmetric $D$-optimal designs \cite{E60,W}.
The table of the number of non-isomorphic tight $2$-codes in $\Omega(d)$ for $d\leq 14$ is obtained by a computer calculation based on Theorem~3.2 in \cite{BC00}.

\section{Results on main eigenvalues}
In this section we give results on main eigenvalues of a Hermitian matrix which 
will be used later. 
Let $H$ be a Hermitian matrix of size $n$ with $s$ distinct eigenvalues $\tau_1<\cdots< \tau_s$.  
Let $E_i$ be the orthogonal projection matrix onto the eigenspace corresponding to $\tau_i$. 
The {\it main angle} $\beta_i$ of $\tau_i$ is defined to be the value 
\[
\beta_i= 
\frac{1}{\sqrt{n}} \sqrt{(E_i \cdot j) ^\ast (E_i \cdot j)},
\]
where 
$j$ is the all-ones vector. 
It is clear that $0\leq \beta_i \leq 1$ and 
$\sum_{i=1}^s \beta_i^2=1$. 

Let $J$ denote the all-ones matrix. 
\begin{lemma}[\cite{NSudapre}]\label{lem:char}
Let $H$ be a Hermitian matrix of size $n$ with $s$ distinct eigenvalues $\tau_1<\cdots <\tau_s$. Let $\beta_i$ be the main angle of $\tau_i$.  
Let $M=H+aJ$, where $a$ is a 
complex
number. 
Then
\[
P_M(x)=P_H(x)\big( 1+a \sum_{i=1}^s  \frac{n \beta_i^2}{\tau_i-x} \big), 
\]
where $P_M$ is the characteristic polynomial of matrix $M$. 
\end{lemma}
An eigenvalue $\tau_i$ is said to be {\it main} 
if $\beta_i\ne 0$.

\begin{theorem} \label{thm:interlace}
Let $H$ be a Hermitian matrix of size $n$, 
and $M=H+a J$, where $a$ is a real number. 
Let $\tau_{1}<\tau_{2}<\cdots<\tau_{r}$ be the distinct main
eigenvalues of $H$, and $\beta_i$ the main angle of $\tau_i$. 
Let $\mu_{1}<\mu_{2}<\cdots<\mu_{s}$ be the distinct 
main eigenvalues of $M$. 
Then $r=s$ holds, and 
\begin{equation} \label{eq:f(x)}
f(x)=\prod_{i=1}^r (\mu_{i}-x)= \prod_{i=1}^r(\tau_{i}-x) 
(1+a \sum_{j=1}^r \frac{n \beta_{j}^2}{\tau_{j}-x}). 
\end{equation} 
 Moreover, if $a>0$, then
$\tau_{1}<\mu_{1}<\tau_{2}< \cdots <\tau_{r}<\mu_{r}$, 
and if $a<0$, then
 $\mu_{1}<\tau_{1}<\mu_{2} <\cdots <\mu_{r}<\tau_{r}$.
\end{theorem}
\begin{proof}

By Lemma~\ref{lem:char}, we have the equality  
\begin{equation}\label{eq:thm2.2} 
\prod_{i=1}^s (\mu_{i}-x)= \prod_{i=1}^r(\tau_{i}-x) 
(1+a \sum_{j=1}^r \frac{n \beta_{j}^2}{\tau_{j}-x}). 
\end{equation} 
By comparing the degrees of the polynomials in both sides, we obtain $s=r$. 

Let $f(x)$ be the polynomial in \eqref{eq:thm2.2}. 
It is easily shown that for $a>0$,
\begin{align*}
f(\tau_{i})>0, &\text{ if $i \equiv 1 \mod 2$,} \\
f(\tau_{i})<0, &\text{ if $i \equiv 0 \mod 2$,} \\
\lim_{x \rightarrow \infty} f(x)<0, &\text{ if $r \equiv 1 \mod 2$,} \\
\lim_{x \rightarrow \infty}f(x)>0, &\text{ if $r \equiv 0 \mod 2$.} 
\end{align*}
This implies that $\tau_{1}<\mu_{1}<\tau_{2}< \cdots <\tau_{r}<\mu_{r}$. 
By the same manner for 
$H=M-aJ$ with $a<0$, 
we can show $\mu_{1}<\tau_{1}<\mu_{2} \cdots <\mu_{r}<\tau_{r}$. 
\end{proof}

\section{Representations of a tournament} \label{sec:Crep}

In this section, we determine $\Rep(G)$ by the multiplicity of the smallest or second-smallest eigenvalue of the Seidel matrix of $G$.  
Let $G=(V,E)$ be a tournament with $n$ vertices. 
The {\it adjacency matrix} $A$ of $G$ is the matrix indexed by the vertex set $V$, with entries given by 
\[
A_{xy}=
\begin{cases}
1 \text{ if $(x,y) \in E$}, \\
0 \text{ otherwise}.
\end{cases}
\]
The Gram matrix of a representation of $G$, with adjacency matrix $A$, can be expressed by
\[
 \alpha A+\overline{\alpha} A^T-\tau I,
\]
where $\alpha$ is an imaginary number, and $\tau$
is a negative real number. Note that 
$\tau$ should be the smallest eigenvalue of $\alpha A+\overline{\alpha} A^T$ to minimize the rank. 
To determine $\Rep(G)$, we will consider $\alpha$ for which 
the multiplicity of the smallest eigenvalue of $\alpha A+\overline{\alpha} A^T$ is maximum.  

\begin{theorem} \label{thm:dim_tou}
Let $G$ be a tournament with $n$ vertices, and $A$ the adjacency matrix. 
Let $\tau_1<\tau_2<\cdots<\tau_s$ be the distinct eigenvalues of 
$S=\sqrt{-1}(A-A^T)$, $\beta_i$ the main angle of $\tau_i$, and 
$m_i$ the multiplicity of $\tau_i$. 
 Let $\alpha$ be the angle with ${\rm Im}( \alpha)>0$ of the representation of $G$ in $\Omega(\Rep(G))$. 
Then the following hold.
\begin{enumerate}
\item If $\beta_1=0$, then $\Rep(G)=n-m_1-1$, and
$\alpha=
(1-c_1\sqrt{-1})/(1+c_1\tau_1)$, 
where $c_1=\sum_{i=2}^s  n\beta_i^2/(\tau_i-\tau_1)$.
\item If $\beta_1 \ne 0$, and $m_1>1$, then $\Rep(G)=n-m_1$,
and $\alpha=-\sqrt{-1}/\tau_1$. 
\item If 
$m_1=1$, $\beta_2=0$, 
and $c_2<0$,
then $\Rep(G)=n-m_2-1$, and
$\alpha=
(1-c_2\sqrt{-1})/(1+c_2\tau_2) $,  
where $c_2=n\beta_1^2/(\tau_1-\tau_2)+\sum_{i=3}^s  n\beta_i^2/(\tau_i-\tau_2)$.
\item Otherwise $\Rep(G)=n-1$. 
\end{enumerate} 
\end{theorem}
\begin{proof}
For $\alpha'=a+\sqrt{-1}$ with $a\in\R$, we have
\[
\alpha' A+ \overline{\alpha'}A^T= aJ+\sqrt{-1} (A-A^T)-a I.
\]
The multiplicity of the smallest eigenvalue of 
$\alpha' A+ \overline{\alpha'}A^T$ is equal to that of 
$M=aJ+\sqrt{-1} (A-A^T)$. 
We would like to find $a \in \mathbb{R}$ such that 
the multiplicity of the smallest eigenvalue of $M$ is maximum. 
Let $\tau_{k_1}<\cdots < \tau_{k_r}$ be the 
distinct main eigenvalues of $S$, 
and $\mu_{l_1}<  \cdots< \mu_{l_r}$ those of $M$. Let $f(x)$ be the polynomial defined as in Theorem~\ref{thm:interlace}. 

(1) By $\beta_1=0$, we have $\tau_1<\tau_{k_1}$.
We would like to find $a \in \mathbb{R}$ such that 
$\mu_{l_1}=\tau_{1}$. For such $a$, the multiplicity of the 
smallest eigenvalue $\tau_1$ 
of $M$ is maximum, and equal to $m_1+1$. 
By Theorem~\ref{thm:interlace}, $\mu_{l_1}=\tau_1$ 
if and only if $f(\tau_1)=0$, namely, 
$a=-1/c_1$. Therefore $\Rep(G)=n-m_1-1$ for $a=-1/c_1$.
By rescaling the diagonal entries of $\alpha' A+ \overline{\alpha'}A^T-(\tau_1-a) I$ to $1$, we obtain $\alpha=
(1-c_1\sqrt{-1})/(1+c_1\tau_1)$.

(2) Since $\beta_1 \ne 0$, we have 
$\tau_1 =\tau_{k_1} \ne \mu_{l_1}$ by 
Theorem~\ref{thm:interlace}. 
Therefore, if $a\ne 0$, the multiplicity of the smallest eigenvalue of $M$ is at most $m_1-1$. 
Thus, for $a=0$, the multiplicity of the smallest eigenvalue of $M$ is maximum, and equal to $m_1$. 
Hence $\Rep(G)=n-m_1$, and $\alpha=-\sqrt{-1}/\tau_1$.

(3) By $c_2<0$, we have $\beta_1>0$ and $\tau_1$ is 
a main eigenvalue.
We would like to find $a \in \mathbb{R}$ such that 
$\mu_{l_1}=\tau_2$.
For such $a$, the multiplicity of the smallest eigenvalue 
$\tau_2$ of $M$ is maximal, and it is $m_2+1$. 
By Theorem~\ref{thm:interlace}, $\mu_{l_1}=\tau_2$ 
if and only if $f(\tau_2)=0$ and $a>0$, namely, 
$a=-1/c_2$ and $c_2<0$.
Therefore we obtain $\Rep(G)=n-m_2-1$, and
 $\alpha=
(1-c_2\sqrt{-1})/(1+c_2\tau_2)$.

(4) 
If $a=0$ and $m_1=1$, then 
the multiplicity of the smallest eigenvalue of 
$M$ is clearly $1$. 

Suppose $\beta_1 \ne 0$, $m_1=1$, $\beta_2=0$, and 
$c_2\geq 0$.  
If $a>0$ holds, then 
$\mu_{l_1}<\tau_2$ by $f(\tau_2)<0$ and $\lim_{x\rightarrow -\infty} f(x)>0$.
If $a<0$ holds, then $\mu_{l_1}<\tau_1$ by Theorem~\ref{thm:interlace}. 
The multiplicity of the smallest eigenvalue $\mu_{l_1}$ of $M$ is $1$. 

Suppose  $\beta_1 \ne 0$, $m_1=1$, $\beta_2 \ne 0$.
Then for any $a\ne 0$, the multiplicity of the smallest eigenvalue $\mu_{l_1}$ of $M$ is $1$
by Theorem~\ref{thm:interlace}.

From the above facts, $\Rep(G)=n-1$ follows. 
\end{proof} 
Note that the conditions (1)--(4) in Theorem~\ref{thm:dim_tou}  are disjoint. 
A tournament which satisfies the condition 
$(i)$ in Theorem~\ref{thm:dim_tou} is said to be of 
{\it Type~$(i)$}
for $i=1,\ldots,4$. 
There is a tournament of each type. 
Lemmas~\ref{lem:tou2}, \ref{lem:tou3}, and  Remark~\ref{rem:D} give examples of Type (1), (2), and (3), respectively.

\section{Tight complex spherical $2$-codes}\label{sec:tight}
In this section, we give bounds on complex spherical $2$-codes.
 We also characterize the tight $2$-codes and $2$-codes in $\Omega(d)$ with $n=2d$ vertices, where $d$ is odd in terms of doubly regular tournaments, skew Hadamard matrix and some skew symmetric $(0,\pm1)$-matrices including skew-symmetric $D$-optimal designs as an application of Theorem~\ref{thm:dim_tou}.

Let $X$ be a finite subset in $\Omega(d)$ of size $n$ with degree $2$, 
and let $A$ be the adjacency matrix of $X$.
Example~6.3 in \cite{RSX} shows that the following are equivalent:
\begin{enumerate}
\item $|X|=2d+1$.
\item $\{I,A,J-A-I\}$ forms the set of adjacency matrices of a non-symmetric association scheme of class $2$.    
\end{enumerate}
\begin{theorem}\label{thm:tightodd}
Let $X$ be a finite subset in $\Omega(d)$ of size $n$ with degree $2$, 
and let $A$ be the adjacency matrix of $X$.
If $d$ is odd, $|X| \leq 2d+1$ holds. Equality holds if and only if $A$ is the adjacency matrix of a doubly regular tournament. 
\end{theorem}
\begin{proof}
The absolute bound \eqref{eq:2_bound} shows that $|X|\leq 2d+1$ holds.
Example~6.3 in \cite{RSX} shows that equality holds if and only if $\{I,A,J-A-I\}$ forms the set of adjacency matrices of a non-symmetric association scheme of class $2$.
The latter condition is equivalent to the condition that $A$ is the adjacency matrix 
of a doubly regular tournament.
\end{proof}

To prove Theorems~\ref{thm:tighteven}, \ref{thm:almosttightodd}, we need the following lemmas. 
\begin{lemma}\label{lem:tou1}
There exists no tournament $A$ of Type 
$(1)$ with $n=2d$ vertices and the spectrum $\{(-\theta)^{d-1},0^2,(\theta)^{d-1}\}$ where $0<\theta$.
\end{lemma}
\begin{proof}
Suppose that there exists such a tournament with Seidel matrix $S$. 
It holds that
$S j=0$ because $\beta_1=\beta_3=0$ and the remaining eigenvalues are all $0$.
However it does not happen because $n=2d$.
\end{proof}

\begin{lemma}\label{lem:tou2}
Let $d$ be an integer at least $3$.
Let $A$ be the adjacency matrix of a tournament of Type 
$(1)$ with $n=2d$ vertices and the spectrum $\{(-\theta)^{d-1},(-\phi)^{1},(\phi)^{1},(\theta)^{d-1}\}$ where $0<\phi<\theta$. 
Then $d$ is odd and $A$ is the adjacency matrix of an induced subgraph of a doubly regular tournament by deleting a vertex.  
\end{lemma}
\begin{proof}
Since the entries of $S^2$ are integers, the eigenvalues of $S^2$ are algebraic integers.
Therefore $\theta^2$ and $\phi^2$ are integer because their multiplicities $2d-2$ and $2$ are different. 
From taking the trace of $S^2$, it follows that the possibility of $(\theta^2,\phi^2)$ is $(2d+1,1)$ or $(2d,d)$.

For the first case, 
$A$ is the adjacency matrix of an induced subgraph of a doubly regular tournament by deleting a vertex \cite[Theorem 1.1]{NSudapre}.  
Thus $n+1=2d+1$ must be congruent to $3$ modulo $4$, which implies that $d$ is odd.

For the second case, consider $\theta^2 I-S^2$.
Since $\theta^2 I-S^2$ is positive semidefinite and the diagonal entries are all $1$, 
the absolute value of an off-diagonal entry of this matrix must be at most $1$.
In fact they must be zero because the size of the matrix $\theta^2 I-S^2$ is even.
Therefore $S^2=(\theta^2-1) I$, which contradicts the fact that $S^2$ has the other eigenvalue $\phi^2$.
\end{proof}

\begin{lemma}\label{lem:tou3}
Let $A$ be the adjacency matrix of a tournament of Type 
$(2)$ with $n=2d$ vertices and the spectrum $\{(-\theta)^{d},(\theta)^{d}\}$ where $0<\theta$. 
Then $d$ is even and $I+A-A^T$ is a skew Hadamard matrix.
\end{lemma}
\begin{proof}
The fact that  $I+A-A^T$ is a skew Hadamard matrix follows from direct calculation, and thus $d$ must be even. 
\end{proof}

\begin{lemma}\label{lem:tou4}
Let $A$ be the adjacency matrix of a tournament of Type 
$(3)$ with the spectrum $\{(-\theta)^{1},(-\phi)^{d-1},(\phi)^{d-1},(\theta)^{1}\}$ where $0<\phi<\theta$. 
Then $d$ is odd and the Seidel matrix $S$ satisfies that $S^2$ is permutaionally similar to 
\begin{align}\label{eq:seidel}
\begin{pmatrix}
k I+l J&0\\0&k I+l J
\end{pmatrix}, 
\end{align}
for some positive integers $k,l$.
\end{lemma}
\begin{proof}
By the condition of Type (3),  $\beta_2=\beta_3=0$ and $\beta_1=\beta_4=1/\sqrt{2}$ hold.
Consider the eigenspaces of $S^2-\phi^2 I$. 
The main angle condition of $S$ implies that the all-ones vector is an eigenvector of $S^2-\phi^2 I$ corresponding to the  eigenvalue $\theta^2-\phi^2$. 
Since the multiplicity of $\theta^2-\phi^2$ is two, let $x$ be the remaining normalized real eigenvector orthogonal to $j$.
Then it holds that 
\begin{align*}
S^2=\phi^2 I+(\theta^2-\phi^2)((1/n)J+x x^T).
\end{align*}
Comparing the diagonal entries, we observe that $n-1=\phi^2+(\theta^2-\phi^2)(1/n+x_i^2)$ for each $i$, where $x_i$ is the $i$-th entry of $x$. 
This implies that $x_i^2$ is independent of the choice of $i$. 
Since the vector $x$ is normalized, we obtain $x_i=\pm1/\sqrt{n}$.
The assumption that $x$ is orthogonal to the all-ones vector shows that each $\pm1/\sqrt{n}$ appears in the entries of $x$ exactly same times.
After some permutation of entries, we may assume that  the first half entries of $x$ are $1/\sqrt{n}$ which means $S^2$ has the form 
\begin{align*}
S^2=\begin{pmatrix}\phi^2 I+\frac{2(\theta^2-\phi^2)}{n}J&0\\0&\phi^2 I+\frac{2(\theta^2-\phi^2)}{n}J   \end{pmatrix}.
\end{align*}  
Since a vector $S(j+\sqrt{n}x)$ is written as a linear combination of $j,x$ and $S=\sqrt{-1}(2A-J+I)$, we have
\begin{align*}
A\begin{pmatrix}j\\0\end{pmatrix}=\begin{pmatrix}a j\\b j\end{pmatrix}
\end{align*}
for some $a,b$. 
Letting $A_1$ be the principal submatrix of $A$ lying the first $d$ rows and columns, 
then  $A_1j=aj$, namely $A_1$ is the adjacency matrix of a regular tournament of order $d$. 
This implies $d$ must be odd.
\end{proof}

\begin{lemma}\label{lem:n=2d}
Let $X$ be a finite subset in $\Omega(d)$ with degree $2$ and size $n=2d$.
The possibilities of the spectrum of $S=\sqrt{-1}(A-A^T)$ are as follows:
\begin{enumerate}
\item[
$({\rm i})$] $X$ is of Type 
$(1)$ with the spectrum $\{(-\theta)^{d-1},0^2,(\theta)^{d-1}\}$.
\item[
$({\rm ii})$] $X$ is of Type 
$(1)$ with the spectrum $\{(-\theta)^{d-1},(-\phi)^{1},(\phi)^{1},(\theta)^{d-1}\}$ with $0<\phi<\theta$.
\item[
$({\rm iii})$] $X$ is of Type 
$(2)$ with the spectrum $\{(-\theta)^{d},(\theta)^{d}\}$.
\item[
$({\rm iv})$] $X$ is of Type 
$(3)$ with the spectrum $\{(-\theta)^{1},(-\phi)^{d-1},(\phi)^{d-1},(\theta)^{1}\}$ with $0<\phi<\theta$.
\end{enumerate}
\end{lemma}
\begin{proof}
Follows from Theorem~\ref{thm:dim_tou}.
\end{proof}

\begin{theorem}\label{thm:tighteven}
Let $X$ be a finite subset of $\Omega(d)$ of size $n$ with degree $2$, 
and let $A$ be the adjacency matrix of $X$.
If $d$ is even, $|X| \leq 2d$ holds. Equality holds if and only if $I+A-A^T$ is a skew Hadamard matrix.
\end{theorem}
\begin{proof}
A necessary condition for the existence of doubly regular tournaments is $|X|\equiv 3 \pmod{4}$, namely $d$ is odd.
Therefore if $d$ is even then $|X|<2d+1$, that is, $|X|\leq 2d$ holds.

Let $H$ be a skew Hadamard matrix of size $n$.
Then $n$ must be a multiple of $4$.
Define $S=\sqrt{-1}(H-I)$ and $A=\frac{1}{2}(-\sqrt{-1}S+J-I)$.
Then the spectrum of $S$ is $\{(-\sqrt{n-1})^{n/2},(\sqrt{n-1})^{n/2}\}$.
Thus $A$ is of Type (2) and the minimum embedding dimension is $d=n/2$.
Therefore $n=2d$.

Let $X$ be a finite subset of $\Omega(d)$ with degree $2$ and size $n=2d$. 
First we consider the case $d=2$.
In this case, the classification of tournaments of order $4$ is given \cite{Mweb} and the list of $A$ are 
\begin{enumerate}
\item[(a)] $\begin{pmatrix}
 0 & 1 & 1 & 1 \\
 0 & 0 & 1 & 1 \\
 0 & 0 & 0 & 1 \\
 0 & 0 & 0 & 0
\end{pmatrix}$ with $\Rep(G)=3$, \qquad (b) $\begin{pmatrix}
 0 & 1 & 1 & 1 \\
 0 & 0 & 0 & 1 \\
 0 & 1 & 0 & 0 \\
 0 & 0 & 1 & 0
\end{pmatrix}$ with $\Rep(G)=2$,
\item[(c)] $\begin{pmatrix}
 0 & 0 & 1 & 1 \\
 1 & 0 & 1 & 0 \\
 0 & 0 & 0 & 1 \\
 0 & 1 & 0 & 0
\end{pmatrix}$ with $\Rep(G)=3$, \qquad (d) $\begin{pmatrix}
 0 & 0 & 1 & 1 \\
 1 & 0 & 0 & 1 \\
 0 & 1 & 0 & 1 \\
 0 & 0 & 0 & 0
\end{pmatrix}$ with $\Rep(G)=2$.
\end{enumerate}
The tournaments (b) and (d) satisfy $n=2d$, and in these cases, $I+A-A^T$ is a skew Hadamard matrix.

Next we consider the case where $d\geq4$.
By 
Lemmas~\ref{lem:tou1}--\ref{lem:n=2d} and the assumption that $d$ is even, $I+A-A^T$ is a skew Hadamard matrix as desired. 
\end{proof}

\begin{theorem}\label{thm:almosttightodd}
Let $d$ be an odd integer at least $3$. 
Let $X$ be a finite subset of $\Omega(d)$ of size $n$ with degree $2$, 
and let $A$ be the adjacency matrix of the tournament obtained from $X$.
The finite subset $X$ has the size $n=2d$ if and only if one of the following occurs:
\begin{enumerate}
\item[{\rm (i)}] $A$ is the adjacency matrix of an induced subgraph of a doubly regular tournament by deleting a vertex.  
\item[{\rm (ii)}] the Seidel matrix $S$ satisfies that $S^2$ is permutaionally similar to 
\begin{align}\label{eq:seidel}
\begin{pmatrix}
k I+l J&0\\0&k I+l J
\end{pmatrix}, 
\end{align}
for some positive integers $k,l$.
\end{enumerate}
\end{theorem}
\begin{proof}
Let $A$ be the adjacency matrix of an induced subgraph of a doubly regular tournament by deleting a vertex.
From Theorem 1.1 and Remark 2.8 in \cite{NSudapre} $A$ is of Type (1) and the minimum embedding dimension is $d=n/2$.
Therefore $n=2d$. 

Let $S$ be the Seidel matrix which satisfies (\ref{eq:seidel}).
By the block form of $S^2$, the eigenvalues $S^2$ are $k+l d,k$ with multiplicities $2,2d-2$ respectively.
Thus the eigenvalues of $S$ are $\pm\sqrt{k+l d},\pm\sqrt{k}$ with multiplicities $1,d-1$ respectively.
The eigenvectors of $S^2$ corresponding to $k+ld$ are the all-ones vector and the $(\pm1)$-vector with the first $d$ entries equal to $1$ and the last $d$ entries equal to $-1$.
This implies that main angles of $S$ corresponding to $\pm \sqrt{k}$ are $0$. 
Thus the adjacency matrix of $S$ is of Type (3) and the minimum embedding dimension $d=n/2$.
Therefore $n=2d$.

Let $X$ be a finite subset in $\Omega(d)$ with degree $2$ and size $n=2d$.
By 
Lemmas~\ref{lem:tou1}--\ref{lem:n=2d} and the assumption that $d$ is odd,  either $A$ is the adjacency matrix of an induced subgraph of a doubly regular tournament by deleting a vertex or  the Seidel matrix $S$ satisfies that $S^2$ is permutaionally similar to \eqref{eq:seidel} as desired.
\end{proof}
\begin{remark} \label{rem:D}
Chadjipantelis and Kounias \cite[Theorem]{CK} showed that supplementary difference sets construct $(\pm1)$-matrix $S$ satisfying \eqref{eq:seidel}. 

For the Seidel matrix $S$ satisfying (\ref{eq:seidel}) with $(k,l)=(n-3,2)$, $\sqrt{-1}S+I$ is known as the $D$-optimal designs \cite{E60,W}. 
Let $A_1,A_2$ be the adjacency matrices of doubly regular tournaments of same order. 
Then a tournament of the adjacency matrix 
\begin{align*}
\begin{pmatrix}
A_1&J\\0&A_2
\end{pmatrix}
\end{align*}
satisfies (\ref{eq:seidel}) for $(k,l)=(d,d-1)$.
For $d=2$, this example corresponds to a skew $D$-optimal design.  
\end{remark}

When $d$ is odd, the number of tight $2$-codes in $\Omega(d)$
is equal to that of doubly regular tournaments of order $2d+1$. 
When $d$ is even, the number of tight $2$-codes 
in $\Omega(d)$ is that of tournaments in the switching classe of the tournament obtained by adding one vertex with no outward edges and all possible inward edges to a doubly regular tournament. 
If we use a computer, the number of non-isomorphic tournaments in a switching class can be calculated by Theorem~3.2 in \cite{BC00}. 
Therefore if doubly regular tournaments are classified, then 
we can determine the number of tight $2$-codes. 
Doubly regular tournaments have been classified for order at most 
$27$ \cite{S95}, and we can find the catalogue in \cite{Mweb}. 
Note that non-isomorphic doubly regular tournaments may be in the same switching
class. 
By using a computer calculation based on Theorem~3.2 in \cite{BC00}, we can give the number of tight $2$-codes as Table 1.  
\begin{center}
	{\small $
\begin{array}{c|cccccccccccccc}
	d   & 1 & 2 & 3 &4&5&6&7&8&9&10&11&12&13&14  \\ \hline 
 |X| & 3 & 4 & 7 &8& 11&12&15&16&19&20&23&24&27&28   \\
	\# & 1 & 2 & 1 &4& 1&8&
 2& 240&2& 8956&37& 11339044& 722&9897616700
\end{array}
	$}
	
Table 1: Tight complex $2$-code $X$ in $\Omega(d)$
\end{center}

\bigskip

\noindent
\textbf{Acknowledgments.} The authors would like to thank the anonymous referees for the useful comments. 
Hiroshi Nozaki is supported by JSPS KAKENHI Grant Numbers 25800011, 26400003. 
Sho Suda is supported by JSPS KAKENHI Grant Numbers 15K21075, 26400003.

\end{document}